\documentclass[11pt, a4paper]{article}

\usepackage{amsmath, amssymb, amsthm, xspace, a4wide, amscd}
\usepackage[english]{babel}
\usepackage[T1]{fontenc}
\usepackage[latin1]{inputenc}

\theoremstyle{plain}
\newtheorem{theorem}{Theorem}[section]
\newtheorem{lemma}[theorem]{Lemma}
\newtheorem{corollary}[theorem]{Corollary}

\theoremstyle{definition}
\newtheorem{definition}[theorem]{Definition}

\theoremstyle{remark}
\newtheorem{remark}{Remark}

\newtheorem{assumption}{Assumption A\hspace*{-4pt}}

\begin{document}

\begin{center}\Large
Averaging principle for equation driven by a stochastic measure\footnote{The final version will be published in ''Stochastics''}
\end{center}

\begin{center}
Vadym Radchenko
\end{center}

\emph{2010 Mathematics Subject Classification}: 60H05; 60H10; 60G57

\emph{Keywords}: Averaging principle; symmetric integral; stochastic measure; stochastic ordinary differential equations; Doss--Sussmann transformation

\begin{abstract}
Equation with the symmetric integral with respect to stochastic measure is considered. For the integrator, we assume only $\sigma$-additivity in probability and continuity of the paths. It is proved that the averaging principle holds for this case, the rate of convergence to the solution of the averaged equation is estimated.
\end{abstract}

\section{Introduction}\label{scintr}

Averaging is an important method to describe the main part of the behavior of dynamical systems. It allows to avoid the detailed analysis of fast-changing variables and consider the simplified equations. This approach is well developed for deterministic and stochastic systems.

Averaging principle for non-random equations is considered in details, for example, in~\cite{murd07}. The stochastic case was studied mainly for equations driven by Wiener process. Averaging of equations was considered in~\cite{samo07}, slow-fast systems -- in \cite{breh12}, \cite[Section~7.9]{frewen}, \cite{wang12}.

Other stochastic integrators were also considered. Averaging of the system with $\alpha$-stable noises was studied in \cite{bao17}, fractional Brownian motion -- in \cite{peis17}, Poisson process -- in \cite{liu12}, \cite{peip17}.

In these papers the strong convergence to the solutions of averaged equations was studied, a similar result is obtained in the given paper. The weak convergence in averaging scheme was considered in~\cite{breh12}, \cite{cerr09},  \cite{fu17}, \cite[Section~II.3]{skor09}.

We will consider averaging of equation driven by general stochastic measure $\mu$. For $\mu$ we assume only $\sigma$-additivity in probability and continuity of the paths. This integrator includes many classes of processes, see examples in Section~\ref{ssstme}. In the previous papers, the scaling invariance of the driving processes was very important in the proofs, we do not assume such property for $\mu$.

In the given paper, the following equation is considered
\begin{equation*}
\circ\,{\rm d}X_{t}^\varepsilon=\sigma(X_t^\varepsilon)\circ\,{\rm d}\mu_t+b(X_t^\varepsilon,t/\varepsilon)\,{\rm d}t,\quad \quad 0\le t\le T,\quad X_{0}^\varepsilon=X_0,
\end{equation*}
where $\circ$ denotes the symmetric integral, defined in~\cite{rads16} (see Section~\ref{sssyin}). We prove that
\begin{equation*}
\sup_{t\in [0,T]}|X_t^\varepsilon-\bar{X}_t|\to 0,\quad \varepsilon\to 0,
\end{equation*}
for $\bar{X}_{t}$ that is the solution to the equation
\begin{equation*}
\circ\,{\rm d}\bar{X}_{t}=\sigma(\bar{X}_t)\circ\,{\rm d}\mu_t+\bar{b}(\bar{X}_t)\,{\rm d}t,\quad 0\le t\le T,\quad \bar{X}_0=X_0.
\end{equation*}
Under some additional assumptions, we obtain the rate of convergence.

The rest of the paper is organized as follows. Section~\ref{scprel} contains the basic facts concerning SMs and symmetric integral. In Section~\ref{scaver} we formulate and prove the main result of the paper (Theorem~\ref{thavsi}).

By $C$ and $C(\omega)$ we will denote positive finite constant and random constant respectively whose exact values are not important.

\section{Preliminaries}\label{scprel}

\subsection{Stochastic measures}\label{ssstme}

Let ${\sf L}_0={\sf L}_0(\Omega, {\mathcal F}, {\sf P} )$ be the set of all real-valued
random variables defined on the complete probability space $(\Omega, {\mathcal F}, {\sf P} )$ (more precisely, the set of equivalence classes). Convergence in ${\sf L}_0$ means the convergence in probability. Let ${\sf X}$ be an arbitrary set and ${\mathcal{B}}$ a $\sigma$-algebra of subsets of ${\sf X}$.

\begin{definition}
A $\sigma$-additive mapping $\mu:\ {\mathcal{B}}\to {\sf L}_0$ is called {\em stochastic measure} (SM).
\end{definition}

We do not assume the moment existence or martingale properties for SM. In other words, $\mu$ is ${\sf L}_0$--valued vector measure.

For deterministic measurable functions $f:{\sf X}\to{\mathbb R}$ the integral $\int_{\sf X} f\,{\rm d}\mu$ is defined. Construction of the integral and basic facts concerning general SMs may be found in~\cite[Chapter 7]{kwawoy}, \cite[Chapter 1]{radmon}. In particular, every bounded measurable $f$ is integrable with respect to any~$\mu$. An analog of the Lebesgue dominated convergence theorem holds for this integral, see \cite[Proposition 7.1.1]{kwawoy}. Some additional facts and review of results about equations driven by SMs are given in \cite{radspr14}.

Important examples of SMs are orthogonal stochastic measures, $\alpha$-stable random measures defined on a $\sigma$-algebra for $\alpha\in (0,1)\cup(1,2]$ (see \cite[Chapter 3]{samtaq}).

Many examples of the SMs on the Borel subsets of $[0,T]$ may be given by the Wiener-type integral
\begin{equation}\label{eqmuax}
\mu(A)=\int_{[0,T]} {\mathbf 1}_A(t)\,{\rm d}X_t.
\end{equation}

We note the following cases of processes $X_t$ in~\eqref{eqmuax} that generate SM.

\begin{enumerate}

\item\label{itmart} $X_t$~-- any square integrable continuous martingale.

\item\label{itfrbr} $X_t=W_t^H$~-- the fractional Brownian motion with Hurst index $H>1/2$, see Theorem~1.1~\cite{memiva}.

\item\label{itsfrb} $X_t=S_t^k$~-- the sub-fractional Brownian motion for $k=H-1/2,\ 1/2<H<1$, see Theorem 3.2~(ii) and Remark 3.3~c) in~\cite{tudor09}.

\item\label{itrose} $X_t=Z_H^k(t)$~-- the Hermite process, $1/2<H<1$, $k\ge 1$, see~\cite{tudor07}, \cite[Section 3.1.3]{tudor13}. $Z_H^2(t)$ is known as the Rosenblatt process, see also~\cite[Section~3]{tudor08}.

\end{enumerate}

Main result of this paper will be obtained under the following assumption on $\mu$.

\begin{assumption}\label{assborcont}
$\mu$ is a SM on Borel subsets of $[0,T]$, and the process $\mu_{t}=\mu((0,t])$ has continuous paths on $[0,T]$.
\end{assumption}

Processes $X_t$ in examples \ref{itmart}--\ref{itrose} are continuous, therefore A\ref{assborcont} holds in these cases.

Give an another example. Let $\eta$ be an arbitrary SM defined on Borel subsets of $[a,b]\subset{\mathbb R}$, function $f:[0,T]\times[a,b]\to{\mathbb R}$ be such that $f(0,x)=0$, and
\begin{equation*}
|f(t,x)-f(s,y)|\le L(|t-s|+|x-y|^\gamma),\quad \gamma>1/2,\quad L\in{\mathbb R}.
\end{equation*}
Then $f(\cdot,x)$ is absolutely continuous for each $x$,  $\Bigl|\dfrac{\partial f(t, x)}{\partial t}\Bigr|\le L$ a.~e., and  we can define SM
\begin{equation}\label{eqmuet}
\mu(A)=\int_{[a,b]} \,{\rm d}\eta(x) \int_{A}\dfrac{\partial f(t, x)}{\partial t}\,{\rm d}t,\quad A\in\mathcal{B}([0,T]).
\end{equation}
The $\sigma$-additivity of $\mu$ follows from the analog of the Lebesgue dominated convergence theorem, see details in~\cite[Section~3]{radt06}.
Theorem~1 of \cite{radt06} implies that the process
\begin{equation*}
\mu_t=\mu((0,t])=\int_{[a,b]} f(t, x)\,{\rm d}\eta(x),\quad t\in [0,T],
\end{equation*}
has a continuous version. Thus, in this case the process $X_t=\mu_t$ in~\eqref{eqmuax} defines an SM that satisfies A\ref{assborcont}.

In some propositions  we will impose the following condition.
\begin{assumption}\label{assintm}
There exists a real-valued finite measure~${\sf m}$ on~$\left({\sf X}, {\mathcal{B}}\right)$ with the following property: if a measurable function $h:{\sf X}\to\mathbb{R}$ is such that \mbox{$\int_{\sf X}h^2\,{\rm d}{\sf m}<+\infty$} then $h$ is integrable with respect to~$\mu$ on~${\sf X}$.
\end{assumption}

This assumption holds in examples~\ref{itfrbr}, \ref{itsfrb} for the Lebesgue measure ${\sf m}$ (see \cite{memiva}, \cite{tudor09}), for $\alpha$-stable random measure and the control measure ${\sf m}$ (see (3.4.1)~\cite{samtaq}). If martingale $X_t$ in example~\ref{itmart} has the deterministic characteristic then  A\ref{assintm} is fulfilled for ${\sf m}(A)=\int_A \,{\rm d}\langle X_t\rangle$.

If A\ref{assintm} holds for SM $\eta$ in~\eqref{eqmuet} then it holds for $\mu$. This follows from the boundedness of $\dfrac{\partial f(t, x)}{\partial t}$.

This assumption is used in the following statement.

\begin{lemma}\label{crsumf}  (Corollary~3.3~\cite{rads16})
If A\ref{assintm} holds then the set of random variables
\[
\Bigl\{\sum_{k=1}^{j} \Bigl(\int_{\sf X} f_k\,{\rm d}\mu\Bigr)^2\quad \Bigr|\quad f_k:{\sf X}\to{\mathbb R}\ {\rm are\ measurable},\quad \sum_{k=1}^{j} f_k^2(x)\le 1,\quad j\ge 1\Bigr\}
\]
is bounded in probability.
\end{lemma}

Recall that set of random variables $\xi_\alpha$, $\alpha\in{\mathcal A}$ is bounded in probability if
\begin{equation*}
\sup_{\alpha\in{\mathcal A}} {\sf P}(|\xi_\alpha|\ge c)\to 0,\quad c\to\infty.
\end{equation*}

\subsection{Symmetric integral}\label{sssyin}

The symmetric integral of random functions with respect to stochastic measures was considered in~\cite{rads16}. We review the basic facts and definitions.

\begin{definition} Let $\xi_{t}$ and $\eta_{t}$ be random processes on $[0,T]$, $0=t_{0}^n<t_{1}^n<\dots<t_{j_n}^n=T$ be a sequence of partitions such that $\max_k |t_k^n-t_{k-1}^n|\to 0$, $n\to\infty$. We define
\begin{equation}\label{eqdfis}
\int_{(0,T]}\xi_{t}\circ\,{\rm d}\eta_{t}:={\rm p}\lim_{n\to\infty}\sum_{k=1}^{j_n}\frac{\xi_{t_{k-1}^n}+\xi_{t_{k}^n}}{2}\, (\eta_{t_{k}^n}- \eta_{t_{k-1}^n} )
\end{equation}
provided that this limit in probability exists.
\end{definition}

For Wiener process $\eta_t$ and adapted $\xi_t$ we obtain the classical Stratonovich integral. If $\eta_t$ and $\xi_t$ are H\"{o}lder continuous with exponents $\gamma_\eta$ and $\gamma_\xi$, $\gamma_\eta+\gamma_\xi>1$, then value of~\eqref{eqdfis} equals to the integral defined in~\cite{zahle98}.

The following theorem describes the class of processes for which the integral exists.

\begin{assumption}\label{assv}
$V_{t}$ is a continuous process of bounded variation on $[0,T]$.
\end{assumption}

\begin{theorem} (Theorem~4.6~\cite{rads16})
Let A\ref{assborcont} and~A\ref{assv} hold, $f\in {\mathbb C}^{1,1} ({\mathbb R}^2 )$. Then integral~\eqref{eqdfis} of $f(\mu_{t}, V_{t})$ with respect to $\mu_{t}$ is well defined, and
\begin{equation*}
\int_{(0,T]}f(\mu_{t}, V_{t})\circ\,{\rm d}\mu_{t}
=G(\mu_{t},V_{t})-G(\mu_{0},V_{0})-\int_{(0,T]} G_2'(\mu_{t},V_{t})\,{\rm d}V_{t},\label{eqintf}
\end{equation*}
where $G(x,v)=\int_{0}^x f(y,v)\,{\rm d}y$.
\end{theorem}

We will consider a stochastic equation of the form
\begin{equation}\label{eqstmu}
\circ\,{\rm d}X_{t}=\sigma(X_t)\circ\,{\rm d}\mu_t+b(X_t,t)\,{\rm d}t,\quad 0\le t\le T,
\end{equation}
$X_0$ is a given random variable.

\begin{definition}\label{dfsleq2}  A process $X_{t}$, $0\le t\le T$ is a \emph{solution} to~\eqref{eqstmu} if:

1) $X_{t}=f(\mu_{t},Y_{t})$, $f\in{\mathbb C}^{2,1}({\mathbb R}^2)$, $Y_{t}$ is a continuous process of bounded variation;

2) for any process $Z_s=\psi(\mu_s,X_s)$, $\psi\in {\mathbb C}^{1,1}({\mathbb R}^2)$, we have
\begin{equation*}
\int_{(0,t]} Z_s\circ\,{\rm d}X_s=\int_{(0,t]} Z_s\sigma(X_s)\circ\,{\rm d}\mu_s+\int_{(0,t]}  Z_s b(X_s,s)\,{\rm d}s,\quad t\in [0,T].
\end{equation*}
\end{definition}

For $Z_s\equiv 1$ we get the usual integral form of the stochastic equation. This form of Definition~\ref{dfsleq2}~2) was important for the proof of uniqueness of the solution.

Solution to~\eqref{eqstmu} was obtained in~\cite{rads16} using the Doss--Sussmann transformation.

\begin{assumption}\label{asssigmab}
1) $\sigma\in {\mathbb C}^2({\mathbb R})$ and the derivatives $\sigma'$, $\sigma''$ are bounded;

2) $b\in {\mathbb C}({\mathbb R}^2)$;

3) for each $c>0$ there exists a $L(c)$ such that
\[
|b(x,t)-b(y,t)|\le L(c)|x-y|,\quad |x|,\ |y|\le c;
\]

4) $b$ is bounded.

\end{assumption}

Let $F:{\mathbb R}^2\to{\mathbb R}$ be the solution of the equation
\begin{equation}\label{eqderf}
\frac{\partial F}{\partial  r}(r,x)=\sigma(F(r,x)),\quad F(0,x)=x,
\end{equation}
which exists globally because of our assumptions. Set $H(r,x)=F^{-1}(r,x)$, where the inverse is taken with respect to $x$. We have that $F,\ H\in {\mathbb C}^{2,2}({\mathbb R}^2)$ and
\begin{equation}\label{eqprdf}
\frac{\partial F}{\partial  x}(r,x)=\exp\Bigl(\int_0^r \sigma'(F(s,x))\,{\rm d}s\Bigr)
\end{equation}
(see calculations in~(5.5)--(5.11)~\cite{rusval00}).

\begin{theorem}\label{thsolx}  (Theorem~5.3~\cite{rads16}) Let A\ref{assborcont} and~A\ref{asssigmab} hold, $X_0$ be an arbitrary random variable. Then equation~\eqref{eqstmu} has a unique solution $X_t=F(\mu_t,Y_t)$, where $Y_t$ is the solution of the random equation
\begin{equation}\label{eqsolx}
Y_t=H(0,X_0)+\int_0^t \frac{\partial H}{\partial
x} (\mu_s, F(\mu_s,Y_s)) b(F(\mu_s,Y_s),s)\,{\rm d}s.
\end{equation}
\end{theorem}

Further, we will need Lipschitz properties of functions in~\eqref{eqsolx}.

Using~\eqref{eqderf}, we obtain
\begin{equation*}
\begin{split}
F(r,x)=x+\int_0^r \sigma(F(s,x))\,{\rm d}s
{\Rightarrow} \frac{\partial F(r,x)}{\partial  x}=1+\int_0^r  \sigma'(F(s,x))\frac{\partial F(s,x)}{\partial  x}\,{\rm d}s\\
\stackrel{ A\ref{asssigmab}.1)}{{\Rightarrow}}
\Bigl|\frac{\partial F(r,x)}{\partial  x}\Bigr|\le 1+C\int_0^r  \Bigl|\frac{\partial F(s,x)}{\partial  x}\Bigr|\,{\rm d}s.
\end{split}
\end{equation*}
The Gronwall inequality implies that
\begin{equation*}
\Bigl|\frac{\partial F(r,x)}{\partial  x}\Bigr|\le \exp\{Cr\}.
\end{equation*}

Therefore, for $\mu_s$ with continuous  paths we obtain that
\begin{equation*}
\Bigl|\frac{\partial F(\mu_s,x)}{\partial  x}\Bigr|\le C(\omega).
\end{equation*}

Also from A\ref{asssigmab}.1) and~\eqref{eqprdf} we obtain that
\begin{equation*}
\Bigl|\frac{\partial F(\mu_s,x)}{\partial  x}\Bigr|\ge C(\omega)>0.
\end{equation*}
Therefore,
\begin{equation}\label{eqhmuh}
\Bigl|\frac{\partial H(\mu_s,x)}{\partial  x}\Bigr|\le C(\omega).
\end{equation}
Here $C(\omega)$ do not depend of $x\in{\mathbb R}$,  $s\in[0, T]$.

Note that we use continuous version of process $\mu_s$, solve \eqref{eqsolx} for each fixed $\omega$. Therefore, our estimates hold for all $\omega\in\Omega$.

\section{Averaging principle}\label{scaver}

For each $\varepsilon>0$ consider the equation
\begin{equation}\label{eqstmua}
\circ\,{\rm d}X_{t}^\varepsilon=\sigma(X_t^\varepsilon)\circ\,{\rm d}\mu_t+b(X_t^\varepsilon,t/\varepsilon)\,{\rm d}t,\quad 0\le t\le T,\quad X_{0}^\varepsilon=X_0,
\end{equation}
$X_0$ is an arbitrary random variable.

From Theorem~\ref{thsolx} it follows that $X_t^\varepsilon=F(\mu_t,Y_t^\varepsilon)$, where
\begin{equation}\label{eqsolxe}
Y_t^\varepsilon=H(0,X_0)+\int_0^t \frac{\partial H}{\partial
x} (\mu_s, F(\mu_s,Y_s^\varepsilon)) b(F(\mu_s,Y_s^\varepsilon),s/\varepsilon)\,{\rm d}s.
\end{equation}

Assumption A\ref{asssigmab}.4) and~\eqref{eqhmuh} imply that
\begin{equation}\label{eqyteb}
|Y_t^\varepsilon|\le C(\omega),\quad |X_t^\varepsilon|\le C(\omega),
\end{equation}
where $C(\omega)$ does not depend of $t\in[0, T]$, $\varepsilon>0$.

Assume that there exist the following limit
\begin{equation*}
\bar{b}(y)=\lim_{t\to\infty}\frac{1}{t}\int_0^t b(y,s)\,{\rm d}s.
\end{equation*}
Obviously, $\bar{b}$ satisfies A\ref{asssigmab}.2)-4) (as function of one variable).

\begin{assumption}\label{assaverb} Function
 $G(y,r)=\int_0^{r} (b(y,s)-\bar{b}(y))\,{\rm d}s$, $r\in{\mathbb R}_+,\ y\in{\mathbb R}$ is bounded.
\end{assumption}

This holds, for example, if $b(y,s)$ is periodic in $s$ for each fixed $y$.

Averaged form of~\eqref{eqstmua} is the following
\begin{equation}\label{eqstmuav}
\circ\,{\rm d}\bar{X}_{t}=\sigma(\bar{X}_t)\circ\,{\rm d}\mu_t+\bar{b}(\bar{X}_t)\,{\rm d}t,\quad 0\le t\le T,\quad \bar{X}_{0}^\varepsilon=X_0.
\end{equation}

From Theorem~\ref{thsolx} it follows that $\bar{X}_t=F(\mu_t,\bar{Y}_t)$, where
\begin{equation}\label{eqsolxeav}
\bar{Y}_t=H(0,X_0)+\int_0^t \frac{\partial H}{\partial
x} (\mu_s, F(\mu_s,\bar{Y}_s)) \bar{b}(F(\mu_s,\bar{Y}_s))\,{\rm d}s.
\end{equation}

Note that functions $F$, $H$ are the same in \eqref{eqsolxe} and \eqref{eqsolxeav}.

The main result of the paper is the following.

\begin{theorem}\label{thavsi}
1) Assume that A\ref{assborcont}, A\ref{asssigmab}, and A\ref{assaverb} hold, $X_t^\varepsilon$ and $\bar{X}_t$ are the solutions of~\eqref{eqstmua} and~\eqref{eqstmuav} respectively. Then for each $\omega\in\Omega$
\begin{equation}\label{eqxtbx}
\sup_{t\in [0,T]} |X_t^\varepsilon-\bar{X}_t|\to 0,\quad \varepsilon\to 0.
\end{equation}

2) Let, in addition, A\ref{assintm} holds. Then the set of the random variables
\begin{equation*}
\dfrac{\sup_{t\in [0,T]} |X_t^\varepsilon-\bar{X}_t|}{\varepsilon^{1/3}},\quad \varepsilon>0
\end{equation*}
is bounded in probability.
\end{theorem}

\begin{proof}
Functions $F$, $\frac{\partial H}{\partial x}$ are locally Lipschitz. Using \eqref{eqyteb} we obtain
\begin{equation}\label{eqxtbxl}
\begin{split}
|X_t^\varepsilon-\bar{X}_t|{\le} C(\omega) |Y_t^\varepsilon-\bar{Y}_t|\\
\stackrel{\eqref{eqsolxe},\eqref{eqsolxeav}}{=}C(\omega) \Bigl|\int_0^t \frac{\partial H}{\partial
x} (\mu_s, X_s^\varepsilon) b(X_s^\varepsilon,s/\varepsilon)\,{\rm d}s-\int_0^t \frac{\partial H}{\partial
x} (\mu_s, \bar{X}_s) \bar{b}(\bar{X}_s)\,{\rm d}s\Bigr|\\
\le C(\omega)\Bigl|\int_0^t \frac{\partial H}{\partial x} (\mu_s, X_s^\varepsilon)  (b(X_s^\varepsilon,s/\varepsilon)-{b}(\bar{X}_s,s/\varepsilon) )\,{\rm d}s\Bigr|\\
+ C(\omega)\Bigl|\int_0^t \frac{\partial H}{\partial x} (\mu_s, X_s^\varepsilon)  ({b}(\bar{X}_s,s/\varepsilon) -\bar{b}(\bar{X}_s) )\,{\rm d}s\Bigr|\\
+C(\omega)\Bigl|\int_0^t \Bigl(\frac{\partial H}{\partial x} (\mu_s, X_s^\varepsilon)- \frac{\partial H}{\partial
x} (\mu_s, \bar{X}_s) \Bigr)\bar{b}(\bar{X}_s)\,{\rm d}s\Bigr|\\
\stackrel{\eqref{eqhmuh}}{\le} C(\omega)\int_0^t  |b(X_s^\varepsilon,s/\varepsilon)-{b}(\bar{X}_s,s/\varepsilon)|\,{\rm d}s+ C(\omega)\Bigl|\int_0^t \frac{\partial H}{\partial x} (\mu_s, X_s^\varepsilon)  ({b}(\bar{X}_s,s/\varepsilon) -\bar{b}(\bar{X}_s) )\,{\rm d}s\Bigr|\\
+C(\omega)\sup|\bar{b}|\int_0^t \Bigl|\frac{\partial H}{\partial x} (\mu_s, X_s^\varepsilon)- \frac{\partial H}{\partial x} (\mu_s, \bar{X}_s) \Bigr|\,{\rm d}s\\
\stackrel{A\ref{asssigmab}.3)}{\le}  C(\omega)\int_0^t  |X_s^\varepsilon-\bar{X}_s  |\,{\rm d}s+ C(\omega)\Bigl|\int_0^t \frac{\partial H}{\partial x} (\mu_s, X_s^\varepsilon)  ({b}(\bar{X}_s,s/\varepsilon) -\bar{b}(\bar{X}_s) )\,{\rm d}s\Bigr|.
\end{split}
\end{equation}

Consider the second term of the last sum. Divide $[0,T]$ into $n$ segments of length $\Delta=\frac{T}{n}$. We have
\begin{equation*}
\begin{split}
I:=\Bigl|\int_0^t \frac{\partial H}{\partial x} (\mu_s, X_s^\varepsilon)  ({b}(\bar{X}_s,s/\varepsilon) -\bar{b}(\bar{X}_s) )\,{\rm d}s\Bigr|\\
\le \sum_{k=0}^{n-1} \Bigl|\int_{(k\Delta\wedge t,(k+1)\Delta\wedge t]}   \frac{\partial H}{\partial x} (\mu_s, X_s^\varepsilon)  ({b}(\bar{X}_s,s/\varepsilon) -\bar{b}(\bar{X}_s) )\,{\rm d}s\Bigr|\\
\le\sum_{k=0}^{n-1}  \Bigl( \Bigl|\int_{(k\Delta\wedge t,(k+1)\Delta\wedge t]}   \Bigl(\frac{\partial H}{\partial x} (\mu_s, X_s^\varepsilon)  - \frac{\partial H}{\partial x} (\mu_{k\Delta}, X_{k\Delta}^\varepsilon)\Bigr)({b}(\bar{X}_s,s/\varepsilon) -\bar{b}(\bar{X}_s) )\,{\rm d}s\Bigr|\\
+\Bigl|\int_{(k\Delta\wedge t,(k+1)\Delta\wedge t]}  \frac{\partial H}{\partial x} (\mu_{k\Delta}, X_{k\Delta}^\varepsilon)  ({b}(\bar{X}_s,s/\varepsilon)-{b}(\bar{X}_{k\Delta},s/\varepsilon) )\,{\rm d}s\Bigr| \\
+\Bigl|\int_{(k\Delta\wedge t,(k+1)\Delta\wedge t]}  \frac{\partial H}{\partial x} (\mu_{k\Delta}, X_{k\Delta}^\varepsilon)  ({b}(\bar{X}_{k\Delta},s/\varepsilon)-\bar{b}(\bar{X}_{k\Delta}) )\,{\rm d}s\Bigr|\\
+\Bigl|\int_{(k\Delta\wedge t,(k+1)\Delta\wedge t]}  \frac{\partial H}{\partial x} (\mu_{k\Delta}, X_{k\Delta}^\varepsilon) (\bar{b}(\bar{X}_{k\Delta})-\bar{b}(\bar{X}_s) )\,{\rm d}s\Bigr|\Bigr)\\
:=\sum_{k=0}^{n-1} (I_{1k}+I_{2k}+I_{3k}+I_{4k}),\\
I_{1k}\stackrel{A\ref{asssigmab}.4)}{\le} C \int_{(k\Delta\wedge t,(k+1)\Delta\wedge t]}  \Bigl|\frac{\partial H}{\partial x} (\mu_s, X_s^\varepsilon)  - \frac{\partial H}{\partial x} (\mu_{k\Delta}, X_{k\Delta}^\varepsilon)\Bigl|\,{\rm d}s\\
\le C(\omega) \int_{(k\Delta\wedge t,(k+1)\Delta\wedge t]}  (|\mu_s-\mu_{k\Delta}|+ |X_s^\varepsilon- X_{k\Delta}^\varepsilon|)\,{\rm d}s,\\
I_{2k}\stackrel{A\ref{asssigmab}.3),\eqref{eqhmuh}}{\le} C(\omega)\int_{(k\Delta\wedge t,(k+1)\Delta\wedge t]}  |X_s^\varepsilon- X_{k\Delta}^\varepsilon|\,{\rm d}s,\\
I_{3k}=\Bigl|\frac{\partial H}{\partial x} (\mu_{k\Delta}, X_{k\Delta}^\varepsilon)\Bigr|\Bigl|\int_{(k\Delta\wedge t,(k+1)\Delta\wedge t]}   ({b}(\bar{X}_{k\Delta},s/\varepsilon)-\bar{b}(\bar{X}_{k\Delta}) )\,{\rm d}s\Bigr|\\
\stackrel{\eqref{eqhmuh},s=r\varepsilon}{\le}C(\omega)\varepsilon
\Bigl|\int_{((k\Delta\wedge t)/\varepsilon,((k+1)\Delta\wedge t)/\varepsilon]}   ({b}(\bar{X}_{k\Delta},r)-\bar{b}(\bar{X}_{k\Delta}) )\,{\rm d}r\Bigr|\stackrel{A\ref{assaverb}}{\le}C(\omega)\varepsilon,\\
I_{4k}\stackrel{\eqref{eqhmuh}}{\le} C(\omega) \int_{(k\Delta\wedge t,(k+1)\Delta\wedge t]}  |\bar{X}_s- \bar{X}_{k\Delta}|\,{\rm d}s.
\end{split}
\end{equation*}
We have that ${X}_s^\varepsilon=F(\mu_s, {Y}_s^\varepsilon)$, where $F$ is locally Lipschitz. From \eqref{eqsolx}, \eqref{eqhmuh} and A\ref{asssigmab}.4) it follows that $|{Y}_t^\varepsilon-\bar{Y}_s^\varepsilon|\le C(\omega)|t-s|$. Therefore,
\begin{equation*}
|X_s^\varepsilon- X_{k\Delta}^\varepsilon|\le C(\omega)(|\mu_s-\mu_{k\Delta}|+\Delta),\quad s\in (k\Delta\wedge t,(k+1)\Delta\wedge t].
\end{equation*}
The same estimate holds for $\bar{X}$ in $I_{4k}$.
We arrive at
\begin{equation*}
I\le C(\omega)n\varepsilon+ C(\omega)n\Delta^2+C(\omega) \sum_{k=0}^{n-1} \int_{(k\Delta\wedge t,(k+1)\Delta\wedge t]}  |\mu_s-\mu_{k\Delta}|\,{\rm d}s.
\end{equation*}
Recall that $\Delta=\frac{T}{n}$. Using that $|\mu_s|\le C(\omega)$, consider separately interval $(k\Delta,(k+1)\Delta]\ni t$, and we obtain
\begin{equation}\label{eqjnep}
\begin{split}
I\le J(n,\varepsilon):=C(\omega)n\varepsilon+ C(\omega)\Delta+C(\omega)n\Delta^2+C(\omega) \sum_{k=0}^{n-1} \int_{(k\Delta,(k+1)\Delta]}  |\mu_s-\mu_{k\Delta}|\,{\rm d}s\\
=C(\omega)n\varepsilon+\frac{C(\omega)}{n}+C(\omega)\int_{(0,T]} |\mu_s-\mu_{[s/\Delta]\Delta}|\,{\rm d}s.
\end{split}
\end{equation}
We have used the notation $[x]$ for the greatest integer not exceeding $x$.

From \eqref{eqxtbxl} and our considerations it follows that
\begin{equation*}
|X_t^\varepsilon-\bar{X}_t|{\le}  C(\omega)\int_0^t  |X_s^\varepsilon-\bar{X}_s|\,{\rm d}s+ J(n,\varepsilon).
\end{equation*}
Gronwall inequality implies that for all $t,\varepsilon, n$
\begin{equation}\label{eqgixt}
|X_t^\varepsilon-\bar{X}_t|{\le}  C(\omega) J(n,\varepsilon).
\end{equation}
For each $0<\varepsilon\le 1$ we can take $n=[\varepsilon^{-1/2}]$. Process $\mu_s$ has continuous paths, therefore
\begin{equation*}
\int_{(0,T]} |\mu_s-\mu_{[s/\Delta]\Delta}|\,{\rm d}s\to 0,\quad \Delta\to 0.
\end{equation*}
From \eqref{eqjnep} and \eqref{eqgixt} it follows statement 1) of our theorem.

Now let us prove part 2). We claim that if A\ref{assborcont} and A\ref{assintm} hold then the set of values
\begin{equation*}
\frac{1}{\sqrt{\Delta}}\int_{(0,T]} |\mu_s-\mu_{[s/\Delta]\Delta}|\,{\rm d}s,\quad  \Delta=\frac{T}{n},\quad n\ge 1
\end{equation*}
is bounded in probability. The Cauchy inequality implies that
\begin{equation*}
\int_{(0,T]} \frac{|\mu_s-\mu_{[s/\Delta]\Delta}|}{\sqrt{\Delta}}\,{\rm d}s\le C\Bigl(\int_{(0,T]} \frac{|\mu_s-\mu_{[s/\Delta]\Delta}|^2}{\Delta}\,{\rm d}s\Bigr)^{1/2}.
\end{equation*}
Dividing each $(k\Delta,(k+1)\Delta]$ into segments of length $\alpha=\frac{\Delta}{m}$ for $m$ large enough, we can approximate
\begin{equation*}
\int_{(0,T]} \frac{|\mu_s-\mu_{[s/\Delta]\Delta}|^2}{\Delta}\,{\rm d}s=\sum_{k=0}^{n-1} \int_{(k\Delta,(k+1)\Delta]}\dfrac{|\mu_s-\mu_{k\Delta}|^2}{\Delta}\,{\rm d}s
\end{equation*}
by integral sums
\begin{equation*}
\sum_{k=0}^{n-1} \sum_{j=1}^m \frac{|\mu_{k\Delta+j\alpha}-\mu_{k\Delta}|^2}{\Delta}\alpha=\sum_{k,j}\Bigl(\int f_{k,j}\,{\rm d}\mu\Bigr)^2.
\end{equation*}
Here
\begin{equation*}
f_{k,j}= \sqrt{\frac{\alpha}{\Delta}}{\bf 1}_{(k\Delta,k\Delta+j\alpha]}, \quad \sum_{k,j} f_{k,j}^2\le  \frac{\alpha}{\Delta} m=1.
\end{equation*}
Lemma~\ref{crsumf} implies our claim.

For each $0<\varepsilon\le 1$ we will choose $n=[\varepsilon^{-2/3}]$. From \eqref{eqjnep} it follows that then values of $J(n,\varepsilon)\varepsilon^{-1/3}$ are bounded in probability, applying of~\eqref{eqgixt} finishes the proof.
\end{proof}

\begin{corollary} If conditions of Theorem~\ref{thavsi} 2) hold then for each $\alpha<1/3$
\begin{equation*}
 \dfrac{\sup_{t\in [0,T]}|X_t^\varepsilon-\bar{X}_t|}{\varepsilon^{\alpha}}\stackrel{\sf P}{\to} 0,\quad \varepsilon\to 0.
\end{equation*}
\end{corollary}

\begin{remark} Note that for non-stochastic case $\mu_t=t$ convergence rate in~\eqref{eqxtbx} is $O(\varepsilon)$, see, for example, \cite[Theorem~2.8.1]{murd07}.

Let $\mu_t=W_t$ be the Wiener process. Then A\ref{assintm} holds, equation~\eqref{eqstmu} with Stratonovich integral is equivalent to the following equation with the It\^{o} integral
\begin{equation}\label{eqxtwp}
{\rm d}X_{t}=\sigma(X_t)\,{\rm d}W_t+\hat{b}(X_t,t)\,{\rm d}t,\quad \hat{b}(x,t)=b(x,t)+\dfrac{1}{2}\sigma(x)\sigma'(x).
\end{equation}
Assumption A\ref{assaverb} holds for $b$ iff it is fulfilled for $\hat{b}$. Theorem~\ref{thavsi} 2) directly gives the convergence rate  $O(\varepsilon^{1/3})$ for this It\^{o}-type equation.

We can refine convergence in~\eqref{eqxtbx} if $\mu_t$ is H\"{o}lder continuous with exponent~$\gamma$. In this case
\begin{equation*}
\int_{(0,T]} |\mu_s-\mu_{[s/\Delta]\Delta}|\,{\rm d}s\le C(\omega)\Delta^{\gamma}=C(\omega)n^{-\gamma}.
\end{equation*}
From \eqref{eqjnep} and~\eqref{eqgixt} for $n=[\varepsilon^{-1/(1+\gamma)}]$ we obtain
\begin{equation*}
|X_t^\varepsilon-\bar{X}_t|{\le}  C(\omega) \varepsilon^{\gamma/(1+\gamma)}.
\end{equation*}

We can compare our results with rate of strong convergence obtained for systems of stochastic equations driven by the Wiener processes. In the case of slow-fast system of SDEs in \cite{liu10} was obtained the rate $O(\varepsilon^{1/2})$. For SPDEs in \cite{wang12} it was proved that $\sup_{t\in [0,T]}|X_t^\varepsilon-\bar{X}_t|\varepsilon^{-1/2}$ are bounded in probability. For similar system in \cite{breh12} rate of convergence of~\eqref{eqxtbx} is $O(\varepsilon^{1/2-})$. Note that in these papers non-generated case was studied, and we can not consider \eqref{eqxtwp} as partial case of these systems.

For system of usual stochastic equations driven by Wiener and Poisson processes in~\cite{liu12} was obtained the rate $O(\varepsilon^{1/2})$, for SPDEs in~\cite{peip17} -- $O(\varepsilon^{1/4})$.
\end{remark}

\section*{Acknowledgements}
The author is grateful to Prof. M.~Z\"{a}hle for fruitful discussions during the preparation of this paper and thanks the Friedrich-Schiller-University of Jena for its hospitality.

\section*{Funding}
This work was supported by Alexander von Humboldt Foundation, grant 1074615.

\bigskip

\textsc{Department of Mathematical Analysis, Taras Shevchenko National University of Kyiv, Kyiv 01601, Ukraine}\\
\emph{E-mail adddress}: \verb"vradchenko@univ.kiev.ua"

\end{document}